\documentclass{llncs}

\usepackage{amsmath,amssymb}
\usepackage{graphicx}

\usepackage{mathrsfs}

 \newcommand{\To}{\longrightarrow}
 \newcommand{\Map}[3]{#1\, :\, #2\To #3}

 \newcommand{\Nat}{\mathbb{N}}
 \newcommand{\Int}{\mathbb{Z}}

 \newcommand{\set}[1]{\left\{#1\right\}}
 \newcommand{\Set}[2]{\set{#1\ \vert\ #2}}

\begin{document}

\title{An introduction to coding sequences of graphs}

\author{Shamik Ghosh\thanks{Corresponding author}\inst{1}
\and Raibatak Sen Gupta\inst{1}
\and M. K. Sen\inst{2}}

\institute{
Department of Mathematics, Jadavpur University, Kolkata-700032, India.\\
e-mail: sghosh@math.jdvu.ac.in,\ raibatak2010@gmail.com
\and
Department of Pure Mathematics, University of Calcutta, Kolkata-700019, India.\\
e-mail: senmk6@yahoo.com
}

\date{}
\maketitle

\begin{abstract}
\noindent
In his pioneering paper on matroids in 1935, Whitney obtained a characterization for binary matroids and left a comment at end of the paper that the problem of characterizing graphic matroids is the same as that of characterizing matroids which correspond to matrices (mod 2) with exactly two ones in each column. Later on Tutte obtained a characterization of graphic matroids in terms of forbidden minors in 1959. But it is clear that Whitney indicated about incidence matrices of simple undirected graphs. 

\vspace{1em}\noindent
In this paper, we introduce the concept of a segment binary matroid which corresponds to matrices over $\Int_2$ which has the consecutive $1$'s property (i.e., $1$'s are consecutive) for columns and obtained a characterization of graphic matroids in terms of this.

\vspace{1em}\noindent
In fact, we introduce a new representation of simple undirected graphs in terms of some vectors of finite dimensional vector spaces over $\Int_2$ which satisfy consecutive $1$'s property. The set of such vectors is called a coding sequence of a graph $G$. Among all such coding sequences we identify the one which is unique for a class of isomorphic graphs. We call it the code of the graph. We characterize several classes of graphs in terms of coding sequences. It is shown that a graph $G$ with $n$ vertices is a tree if and only if any coding sequence of $G$ is a basis of the vector space $\Int_2^{n-1}$ over $\Int_2$. 

\vspace{1em}\noindent
Moreover considering coding sequences as binary matroids, we obtain a characterization for simple graphic matroids and found a necessary and sufficient condition for graph isomorphism in terms of a special matroid isomorphism between their corresponding coding sequences. For this, we introduce the concept of strong isomorphisms of segment binary matroids and show that two simple (undirected) graphs are isomorphic if and only if their canonical sequences are strongly isomorphic segment binary matroids. 

\vspace{1em}\noindent
{\small {\bf AMS Subject Classifications:} 05C62, 05C50, 05B35.

\noindent
{\bf Keywords:} Simple undirected graph, graph representation, graph isomorphism, incidence matrix, consecutive $1$'s property, binary matroid, graphic matroid.}
\end{abstract}

\section{Introduction}

There are various representations of simple undirected graphs in terms of adjacency matrices, adjacency lists, incidence matrix, unordered pairs etc. In this paper, we introduce another representation of a simple undirected graph with $n$ vertices in terms of certain vectors in the vector space $\Int_2^{n-1}$ over $\Int_2$. We call the set of vectors representing a graph $G$ as a coding sequence of $G$ and denote it by $\beta(G,n)$. Among all such coding sequences we identify the one which is unique for a class of isomorphic graphs. We call it the code of the graph. We find characterizations of graphs which are connected, acyclic, bipartite, Eulerian or Hamiltonian in terms of $\beta(G,n)$. We prove that a graph $G$ with $n$ vertices is a tree if and only if any coding sequence of $G$ is a basis of the vector space $\Int_2^{n-1}$ over $\Int_2$.

\vspace{1em}\noindent
In his pioneering paper \cite{HW} on matroids in 1935, Whitney left the problem of characterizing graphic matroid open by making the following comment: ``The problem of characterizing linear graphs from this point of view is the same as that of characterizing matroids which correspond to matrices (mod 2) with exactly two ones in each column.'' In 1959, Tutte obtained a characterization of graphic matroids in terms of forbidden minors \cite{TT}. But it is clear that Whitney indicated about incidence matrices of simple undirected graphs. In this paper, use a variation of incidence matrix for the same characterization.

\vspace{1em}\noindent
In section 3, we introduce the concept of a segment binary matroid which corresponds to matrices over $\Int_2$ that has the consecutive $1$'s property (i.e., $1$'s are consecutive) for columns and a characterization of graphic matroids is obtained by considering $\beta(G,n)$ as a segment binary matroid. We introduce the concept of a strong isomorphism for segment binary matroids and show that two simple graphs $G$ and $H$ (with $n$ vertices each) are isomorphic if and only if $\beta (G, n)$ and $\beta (H, n)$ are strongly isomorphic segment binary matroids. 

\vspace{1em}\noindent
For graph theoretic concepts, see \cite{DBW} and for matroid related terminologies, one may consult \cite{Ox}.

\section{Coding sequences}

\begin{definition}
Let $G=(V,E)$ be a simple undirected graph with $n$ vertices and $m$ edges. Let $V=\set{v_0,v_1,\ldots,v_{n-1}}$. Define a map $\Map{f}{V}{\Nat}$ by $f(v_i)=10^i$ and another map $\Map{f^*}{E}{\Nat}$ by $f^*(v_iv_j)=|f(v_i)-f(v_j)|$, if $E\neq\emptyset$. Let $\sigma(G,n)$ be the sequence $\Set{f^*(e)}{e\in E}$ sorted in the increasing order. If $E=\emptyset$, then $\sigma(G,n)=\emptyset$. It is worth noticing that for $e=v_iv_j\in E$, $f^*(e)=|10^i-10^j|$ uniquely determines the pair $(i,j)$ as it is a natural number with $i$ digits, starting with $i-j$ number of $9$'s and followed by $j$ number of $0$'s, when $i>j$. Thus $m$ entries of $\sigma(G,n)$ are all distinct. 

\vspace{1em}\noindent
Now for $E\neq\emptyset$, we define a map $\Map{f^\#}{E}{\Int_2^{n-1}}$ by $f^\#(e)=(x_1,x_2,\ldots,x_{n-1})$, where $x_i=1$, if $(n-i)^{\text{th}}$ digit of $f^*(e)$ from the right is $9$, otherwise $x_i=0$ for $i=1,2,\ldots,n-1$. For convenience we write the field $\Int_2$ as $\set{0,1}$ instead of $\set{\bar{0},\bar{1}}$. Let $\beta(G,n)$ be the sequence $\Set{f^\#(e)}{e\in E}$ sorted in the same order as in $\sigma(G,n)$. If $E=\emptyset$, then $\beta(G,n)=\emptyset$. The sequence $\beta(G,n)$ is called a {\em \textbf{coding sequence}} of the graph $G$. 

\vspace{1em}\noindent
Naturally, $\beta(G,n)$ is not unique for a graph $G$ as it depends on the labeling $f$ of vertices. Now there are $n!$ such labellings and consequently we have at most $n!$ different $\sigma(G,n)$ for a graph $G$. Among which we choose the one, say, $\sigma_c(G,n)$ which is the minimum in the lexicographic ordering of $\Nat^m$. The corresponding $\beta(G,n)$ is called the {\em \textbf{code}} of the graph $G$ and is denoted by $\beta_c(G,n)$. Clearly $\beta_c(G,n)$ is unique for a class of isomorphic graphs with given number of vertices, though it is not always easy to determine generally.
\end{definition}

\begin{example}\label{ex1}
Consider the graph $G$ in Figure \ref{fig1} (left). We have 
$$\sigma(G,4)=(9,90,900,990)\ \text{ and }\ \beta(G,4)=\set{(0,0,1),(0,1,0),(1,0,0),(1,1,0)}$$
according to the labeling of vertices given in Figure \ref{fig1} (left). One may verify that
$$\sigma_c(G,4)=(9,90,99,900) \text{ and } \beta_c(G,4)=\set{(0,0,1),(0,1,0),(0,1,1),(1,0,0)}$$
according to the labeling of vertices shown in Figure \ref{fig1} (right).
\end{example}

\begin{figure}[h]
\begin{center}
\includegraphics*[scale=0.25]{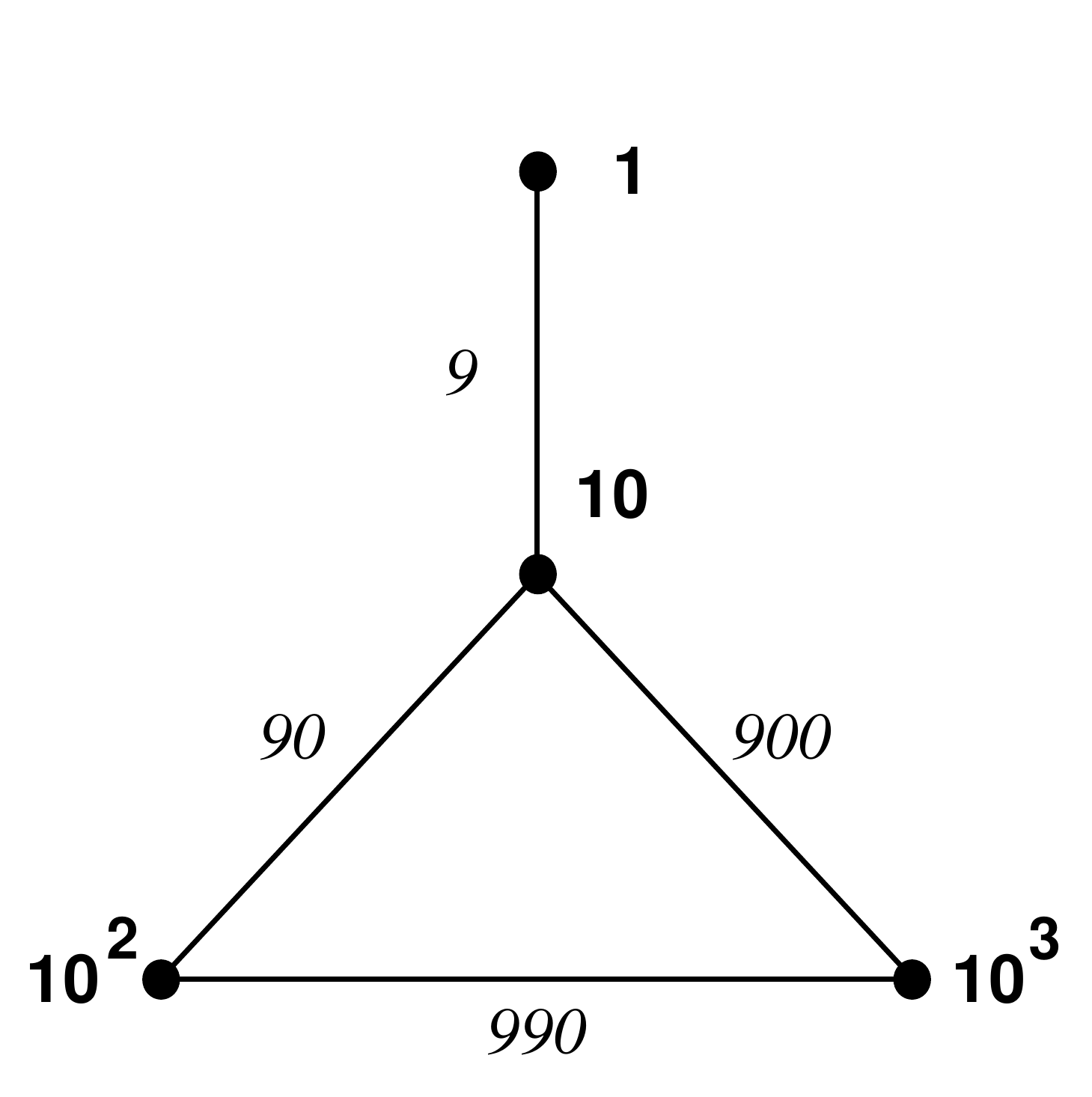}\hspace{1in} \includegraphics*[scale=0.25]{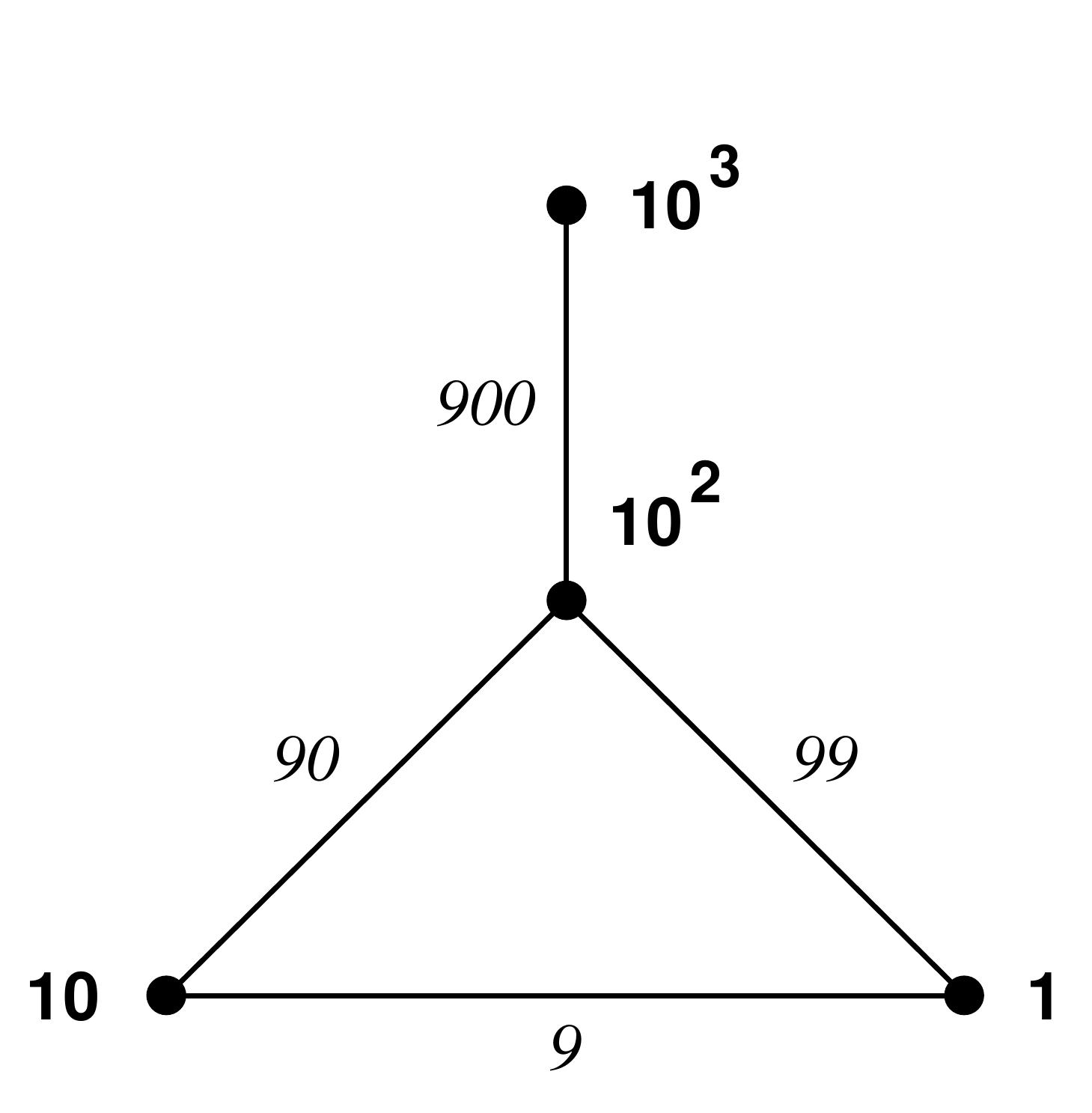}
\caption{The graph $G$ in Example \ref{ex1} with different labellings.}\label{fig1}
\end{center}
\end{figure}

\begin{remark}
An incidence matrix of a (simple undirected) graph $G=(V,E)$ is obtained by placing its vertices in rows and edges in columns and an entry in a row corresponding to a vertex $v$ and in a column corresponding to an edge $e$ of the matrix is $1$ if and only if $v$ is an end point of $e$, otherwise it is $0$. It is important to note that a coding sequence of a graph has a similarity with the incidence matrix of the graph. In fact, given a coding sequence of a graph $G$, one can easily obtain the incidence matrix of $G$ and vice-versa. Also cut-set and circuit subspaces of a vector space of dimension $|E|$ over $\Int_2$ constructed from edges of $G$ are well known \cite{C,D,G}. Further, as we mentioned in the introduction, Whitney expected the characterization of graphic matroids would be obtained from the incidence matrix. Here we consider a variation of it with a consecutive $1$'s representation as it helps us to build a very natural interplay between graph theory, matroids and linear algebra which is evident from Theorems \ref{tree} and \ref{isothm}. 
\end{remark}

\noindent
Throughout this section by a graph we mean a simple undirected graph. 

\begin{definition}\label{seqtogrph}
A non-null vector $e=(x_1,x_2,\ldots,x_{n-1})\in\Int_2^{n-1}$ is said to satisfy the {\em \textbf{consecutive $1$'s property}} if $1$'s appear consecutively in the sequence of coordinates of $e$. Let 
$$C(n-1)=\Set{v\in\Int_2^{n-1}}{v \text{ satisfies the consecutive } 1\text{'s property}}.$$
Clearly, $|C(n-1)|=\binom{n}{2}=\frac{n(n-1)}{2}$ and $G$ is a complete graph with $n$ vertices if and only if $\beta(G,n)=C(n-1)$. In fact, for every $S\subseteq C(n-1)$, there is a unique graph $G(S)$ of $n$ vertices such that $\beta(G,n)=S$. If $S=\emptyset$, then $G$ is the null graph with $n$ vertices and no edges. If $S\neq\emptyset$, each member $e$ of $S$ represents an edge of $G(S)=(V,E)$ with end points $10^{n-i}$ and $10^{n-j-1}$, where the consecutive stretch of $1$'s in $e$ starts from the $i^\text{th}$ entry and ends at the $j^\text{th}$ entry from the left and $V=\set{1,10,10^2,\ldots,10^{n-1}}$. Also it is clear that $C(n-1)\smallsetminus \beta(G,n)$ is a coding sequence of the complement $\bar{G}$ of a graph $G$ with $n$ vertices.

\vspace{0.5em}\noindent
Let $\emptyset\neq S\subseteq C(n-1)$. Let $\tilde{G}(S)$ be the subgraph of $G(S)$ obtained by removing isolated vertices (if any) from $G(S)$. Then $\tilde{G}(S)$ is the subgraph of the complete graph of $n$ vertices induced by the edges represented by the vectors in $S$.
\end{definition}

\noindent
We denote the null vector in the vector space $\Int_2^{n-1}$ by $\mathbf{0}$ for any $n\in\Nat$ and write $\Int_2^0$ for the zero-dimensional space $\set{\mathbf{0}}$. Let $S=\set{e_1,e_2,\ldots ,e_k}\subseteq\Int_2^{n-1}\smallsetminus\set{\mathbf{0}}$, ($k\in\Nat$, $k\leqslant 2^{n-1}$). As $e_i$'s are distinct, we have $e_i+e_j\neq\mathbf{0}$ for all $i\neq j$, $i,j\in\set{1,2,\ldots,k}$. Thus $S$ is a set of linearly dependent vectors in $\Int_2^{n-1}$ over $\Int_2$ if and only if there exists $A\subseteq S$, $|A|\geqslant 3$ such that $\sum\limits_{e\in A} e =\mathbf{0}$. In other words, $S\subseteq\Int_2^{n-1}$ is linearly independent over $\Int_2$ if and only if $S=\emptyset$ or $S=\set{e_1,e_2,\ldots ,e_k}$ for some $k\in\Nat$, $k\leqslant 2^{n-1}$ and $\sum\limits_{e\in A} e \neq\mathbf{0}$ for all $\emptyset\neq A\subseteq S$. We denote the linear span (over $\Int_2$) of a subset $S$ of $\Int_2^{n-1}$ by $\text{Sp}\, (S)$, i.e., $\text{Sp}\, (S)$ is the smallest subspace of $\Int_2^{n-1}$ containing $S$.

\begin{proposition}\label{3cycle}
Let $S=\set{e_1,e_2,e_3}\subseteq C(n-1)$ for some $n\in\Nat$, $n\geqslant 3$. Then $\tilde{G}(S)$ is a $3$-cycle if and only if $e_1+e_2+e_3=\mathbf{0}$.
\end{proposition}

\begin{proof}
First suppose that $\tilde{G}(S)$ is the $3$-cycle shown in Figure \ref{fig3}, where $\alpha,\beta,\gamma\in\set{0,1,\ldots,n-1}$. Without loss of generality we assume $\alpha >\beta >\gamma$. Then
$$\begin{array}{rcl}
e_1 & = & (\underbrace{0,0,\ldots,0}_{n-\alpha -1},\underbrace{1,1,\ldots,1}_{\alpha -\beta},\underbrace{0,0,\ldots,0,0,0,\ldots,0}_{\beta})\\
e_2 & = & (\underbrace{0,0,\ldots,0}_{n-\alpha -1},\underbrace{1,1,\ldots,1,1,1,\ldots,1}_{\alpha -\gamma},\underbrace{0,0,\ldots,0}_{\gamma})\\
e_3 & = & (\underbrace{0,0,\ldots,0,0,0,\ldots,0}_{n-\beta -1},\underbrace{1,1,\ldots,1}_{\beta -\gamma},\underbrace{0,0,\ldots,0}_{\gamma})
\end{array}$$
Clearly $e_1+e_2+e_3=\mathbf{0}$.

\begin{figure}[ht]
\begin{center}
\includegraphics*[scale=0.3]{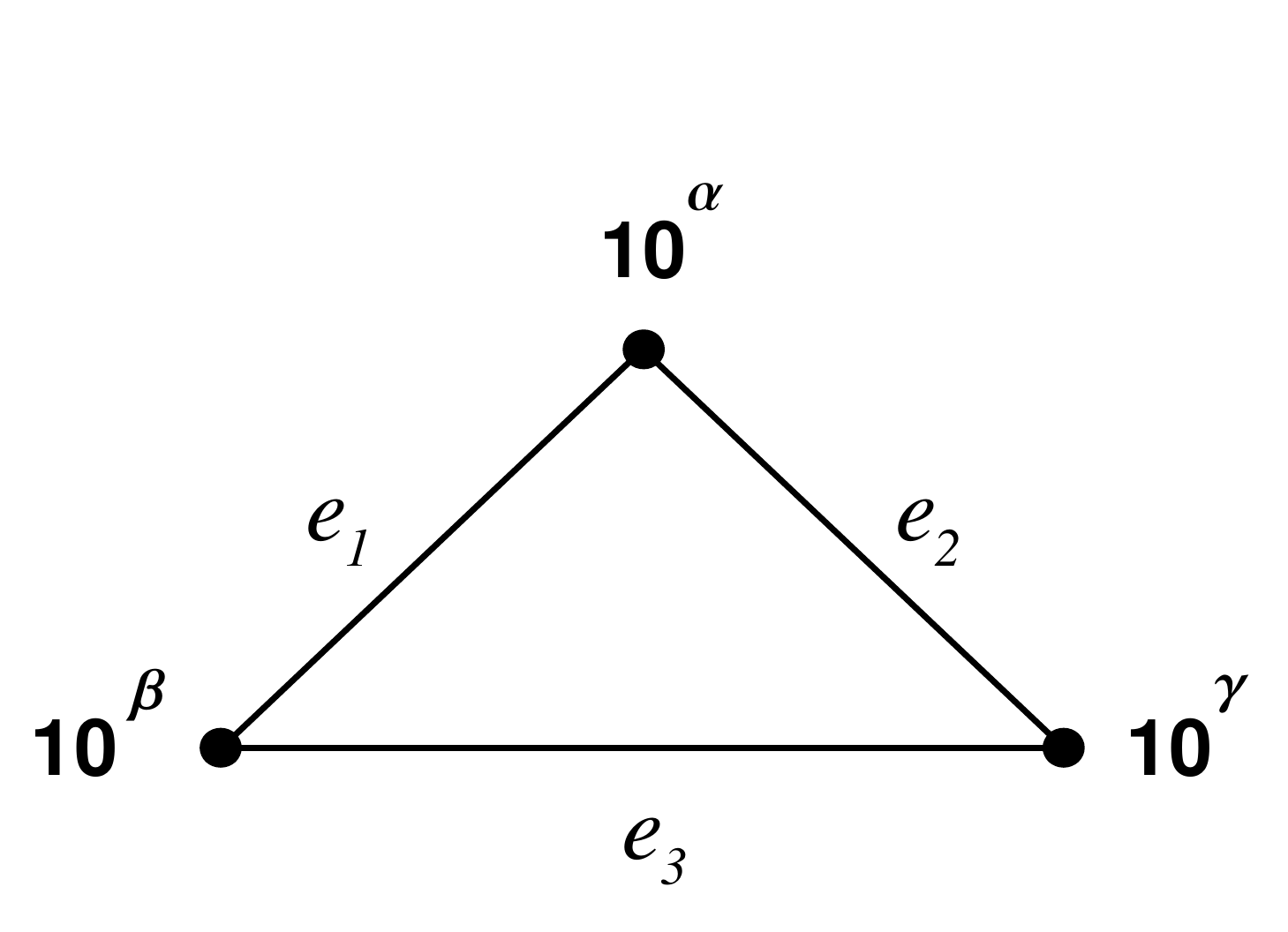}
\caption{A $3$-cycle}\label{fig3}
\end{center}
\end{figure}

\noindent
Conversely, let $e_1+e_2+e_3=\mathbf{0}$. Consider the matrix $M=\left(%
\begin{array}{c}
e_1\\
e_2\\
e_3
\end{array}\right)$, where we represent each $e_i$ as a row matrix consisting of $n-1$ columns for $i=1,2,3$. Since $e_1+e_2+e_3=\mathbf{0}$, in each column where $1$'s appear, they appear exactly in two rows. Let $i$ be the least column number of $M$ that contains $1$. Without loss of generality we assume that $1$'s appear in the first two rows in the $i^\text{th}$ column (otherwise we rearrange rows of $M$). Also suppose that the number of zeros after the stretch of $1$'s in the first row, say, $\beta$ is more than that of the second, say, $\gamma$ (otherwise again we rearrange rows of $M$). Let $\alpha =n-i$. Then the end points of the edge of $\tilde{G}(S)$ corresponding to $e_1$ are $10^\alpha$ and $10^\beta$ and those of the edge corresponding to $e_2$ are $10^\alpha$ and $10^\gamma$. Since $\beta >\gamma$ and $e_1+e_2+e_3=\mathbf{0}$, we have 
$$e_3=e_1+e_2=(\underbrace{0,0,\ldots,0}_{n-\beta -1},\underbrace{1,1,\ldots,1}_{\beta -\gamma},\underbrace{0,0,\ldots,0}_{\gamma}).$$
Thus the end points of the edge of $\tilde{G}(S)$ corresponding to $e_3$ are $10^\beta$ and $10^\gamma$. So the vertices labeled by $10^\alpha,10^\beta$ and $10^\gamma$ form a $3$-cycle with edges corresponding to $e_1,e_2,e_3$, as required.
\end{proof}

\begin{definition}
A set $S\neq\emptyset$ of non-null vectors in $\Int_2^{n-1}$ is called {\em \textbf{reduced}} if $\sum\limits_{e\in A} e \neq\mathbf{0}$ for all $\emptyset\neq A\subsetneqq S$.
\end{definition}

\begin{lemma}\label{kcycle}
Let $S=\set{e_1,e_2,\ldots,e_k}\subseteq C(n-1)$ for some $k,n\in\Nat$, $3\leqslant k\leqslant n$. Then $\tilde{G}(S)$ is a $k$-cycle if and only if $S$ is reduced and $e_1+e_2+\cdots +e_k=\mathbf{0}$.
\end{lemma}

\begin{proof}
We prove by induction on $k$. By Proposition \ref{3cycle}, the result is true for $k=3$. Suppose the result is true for $k=r-1\geqslant 3$. Let $S=\set{e_1,e_2,\ldots,e_r}\subseteq C(n-1)$ for some $r,n\in\Nat$, $3< r\leqslant n$ form the $r$-cycle shown in Figure \ref{fig4} (we renumber $e_i$'s, if necessary). Consider the chord $e$ so that $e_1,e_2$ and $e$ form a triangle. Then $e_1+e_2+e=\mathbf{0}$ by Proposition \ref{3cycle}. So $e=e_1+e_2$. Also $\set{e,e_3,e_4,\ldots,e_r}$ form an $(r-1)$-cycle. So by induction hypothesis $e+e_3+e_4+\cdots+e_r=\mathbf{0}$ which implies $e_1+e_2+e_3+e_4+\cdots+e_r=\mathbf{0}$. Moreover since $S$ is a cycle, no proper subset of $S$ form a cycle. Thus $S$ is reduced by induction hypothesis.

\begin{figure}[ht]
\begin{center}
\includegraphics*[scale=0.35]{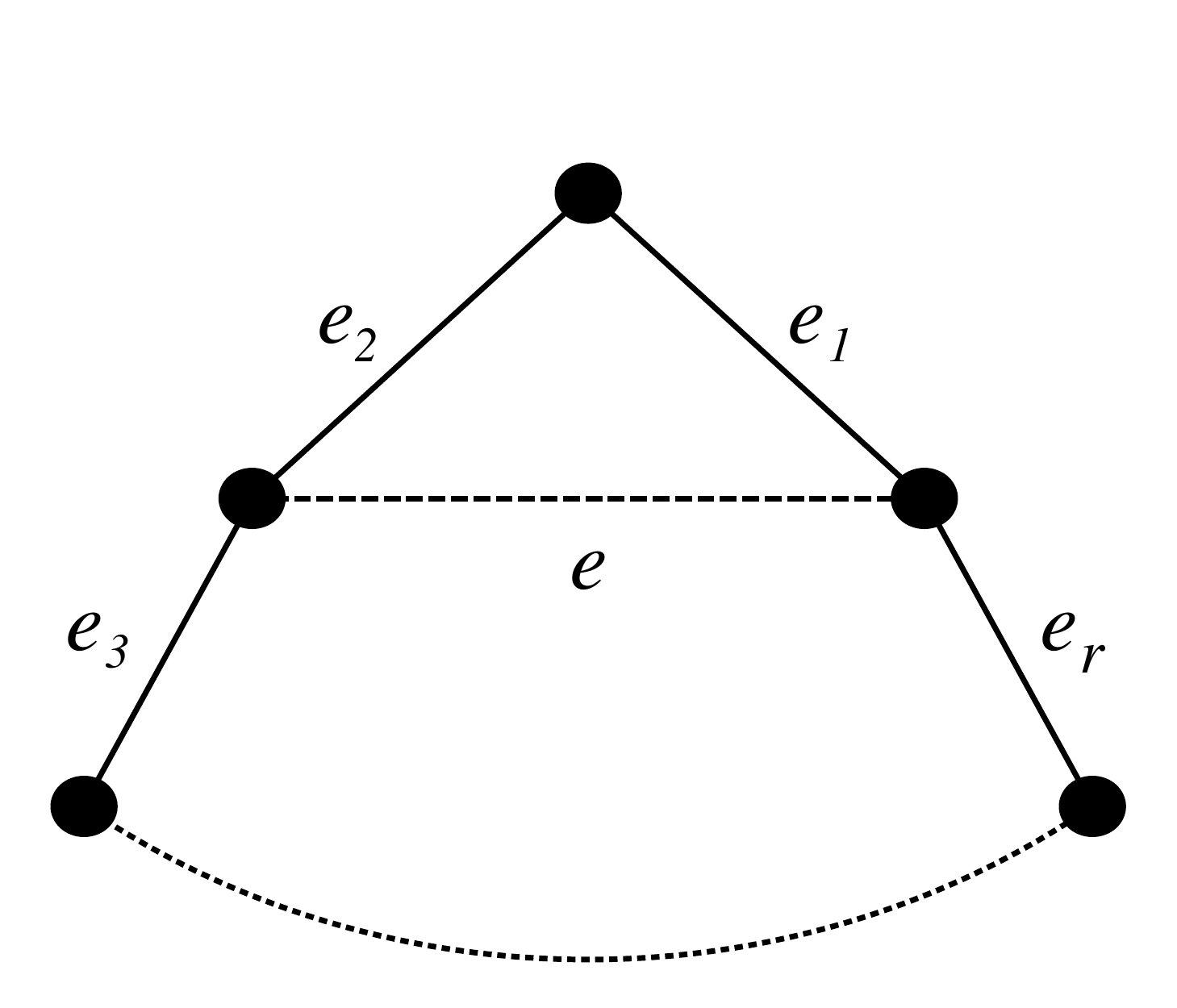}
\caption{An $r$-cycle}\label{fig4}
\end{center}
\end{figure}

\noindent
Conversely, let $S=\set{e_1,e_2,\ldots,e_r}\subseteq C(n-1)$ for some $r,n\in\Nat$, $3< r\leqslant n$ be reduced and $e_1+e_2+\cdots+e_r=\mathbf{0}$. Consider the matrix $M=\left(%
\begin{array}{c}
e_1\\
e_2\\
\cdots\\
e_r
\end{array}\right)$, where we represent each $e_i$ as a row matrix consisting of $n-1$ columns for $i=1,2,\ldots,r$. Let $i$ be the least column number of $M$ that contains $1$. Since the row sum of $M$ is zero (the null vector), the $i^\text{th}$ column contains even number of $1$'s. So there are at least two rows with $1$ in the $i^\text{th}$ column. Rearrange rows of $M$ such that $e_1$ and $e_2$ be two such rows. Since both of these rows begin with $1$ in the $i^\text{th}$ column, the edges corresponding to them have a common end point with label $10^{n-i}$. Join the other end points of $e_1$ and $e_2$ by an edge, say, $e$ to form a triangle with edges $e_1,e_2,e$. Then $e_1+e_2+e=\mathbf{0}$ which implies $e=e_1+e_2$. So $e+e_3+e_4+\cdots+e_r=\mathbf{0}$. 

\vspace{0.5em}\noindent
We claim that $S_1=\set{e,e_3,e_4,\ldots,e_r}$ is reduced. Suppose $A\subsetneqq S_1$, $|A|\geqslant 3$ be such that $a=\sum\limits_{x\in A} x=\mathbf{0}$. If $e\notin A$, then $a\neq 0$ as $S$ is reduced. So $e\in A$. Then replacing $e$ by $e_1+e_2$ in $a$ would again contradict the fact that $S$ is reduced. So $S_1$ is reduced. Hence by induction hypothesis, $S_1$ form an $(r-1)$-cycle. Now replacing the edge $e$ by the path consisting of edges $e_1$ and $e_2$ gives us an $r$-cycle formed by $S$.
\end{proof}

\noindent
The following two corollaries follow immediately from the above lemma.

\begin{corollary}\label{ham}
A graph $G$ with $n\geqslant 3$ vertices $(n\in\Nat)$ is Hamiltonian if and only if for any coding sequence $\beta(G,n)$ of $G$, there exists $S=\set{e_1,e_2,\ldots,e_n}\subseteq\beta(G,n)$ such that $S$ is reduced and $e_1+e_2+\cdots+e_n=\mathbf{0}$.
\end{corollary}

\begin{corollary}\label{acyclic}
A graph $G$ with $n$ vertices $(n\in\Nat)$ is acyclic if and only if any coding sequence $\beta(G,n)$ of $G$ is linearly independent over $\Int_2$.
\end{corollary}

\begin{corollary}\label{eul}
A graph $G$ with at most one non-trivial component and with $n$ vertices $(n\in\Nat)$ is Eulerian if and only if $\displaystyle{\sum\limits_{e\in \beta(G,n)} e =\mathbf{0}}$ for any coding sequence $\beta(G,n)$ of $G$.
\end{corollary}

\begin{proof}
Follows from Lemma \ref{kcycle} and the fact that a circuit in a graph can be decomposed into edge-disjoint cycles.
\end{proof}

\begin{corollary}\label{bip}
A graph $G$ with $n\geqslant 2$ vertices $(n\in\Nat)$ is bipartite if and only if for any coding sequence $\beta(G,n)$ of $G$, $\displaystyle{\sum\limits_{e\in S} e \neq\mathbf{0}}$ for every $S\subseteq \beta(G,n)$ where $|S|$ is odd.
\end{corollary}

\begin{proof}
Follows from Lemma \ref{kcycle} and the fact that a graph is bipartite if and only if it does not contain any odd cycle.
\end{proof}

\begin{theorem}\label{tree}
A graph $G$ with $n$ vertices $(n\in\Nat)$ is a tree if and only if any coding sequence $\beta(G,n)$ of $G$ is a basis of the vector space $\Int_2^{n-1}$ over the field $\Int_2$.
\end{theorem}

\begin{proof}
Suppose $G$ is a tree. Then $G$ is acyclic which implies $\beta(G,n)$ is linearly independent over $\Int_2$ by Corollary \ref{acyclic}. Again since $G$ is a tree, the number of entries in $\beta(G,n)$ is $n-1$, we have $n-1$ linearly independent vectors in $\Int_2^{n-1}$ over $\Int_2$. Thus $\beta(G,n)$ is a basis of $\Int_2^{n-1}$ over $\Int_2$.

\vspace{0.5em}\noindent
Conversely, suppose $\beta(G,n)$ is a basis of $\Int_2^{n-1}$ over $\Int_2$. Then $\beta(G,n)$ is linearly independent over $\Int_2$ and so $G$ is acyclic by Corollary \ref{acyclic}. Also since $\beta(G,n)$ is a basis of $\Int_2^{n-1}$ over $\Int_2$, the number of entries in $\beta(G,n)$ is $n-1$ which implies $G$ has $n-1$ edges. Thus $G$ is a tree.
\end{proof}

\begin{corollary}\label{connected}
A graph $G$ with $n$ vertices $(n\in\Nat)$ is connected if and only if for any coding sequence $\beta(G,n)$ of $G$, $\text{{\em Sp}}\, (\beta(G,n))=\Int_2^{n-1}$.
\end{corollary}

\begin{proof}
We first note that a graph $G$ is connected if and only if $G$ has a spanning tree. Suppose $G=(V,E)$ has a spanning tree $H=(V,E_1)$. Then $\beta(H,n)\subseteq\beta(G,n)$ with the same vertex labeling. But $\beta(H,n)$ is a basis of $\Int_2^{n-1}$ over $\Int_2$ by Theorem \ref{tree}. Thus $\text{Sp}\, (\beta(G,n))\supseteq\text{Sp}\, (\beta(H,n))=\Int_2^{n-1}$. So $\text{Sp}\, (\beta(G,n))=\Int_2^{n-1}$.

\vspace{0.5em}\noindent
Conversely, if $\text{Sp}\, (\beta(G,n))=\Int_2^{n-1}$, then $\beta(G,n)$ contains a basis, say $B$ of $\Int_2^{n-1}$ over $\Int_2$. Then $G(B)$ is a spanning tree of $G$ by Theorem \ref{tree} as $B=\beta(G(B),n)$. Thus $G$ is connected.
\end{proof}

\begin{corollary}\label{spantree}
Let $G$ be a connected graph with $n$ vertices and $S\subseteq \beta(G,n)$. Then $G(S)$ is a spanning tree of $G$ if and only if $S$ is a basis of the vector space $\Int_2^{n-1}$ over the field $\Int_2$.
\end{corollary}

\section{Matroid representation}

Whitney introduced the concept of a matroid in \cite{HW}. There are several ways of defining matroids. We take the one that will serve our purpose. A {\em matroid} $M$ is an ordered pair $(E,\mathscr{B})$ consisting of a finite set $E$ of {\em elements} and a nonempty collection $\mathscr{B}$ of subsets of $E$, called {\em bases} which satisfies the properties: (i) no proper subset of a base is a base and (ii) if $B_1,B_2\in \mathscr{B}$ and $e\in B_1\smallsetminus B_2$, then there exists $f\in B_2\smallsetminus B_1$ such that $(B_1\smallsetminus \set{e})\cup \set{f}\in\mathscr{B}$. {\em Independent sets} of $M$ are subsets of bases and minimal dependent sets are {\em circuits}. The {\em cycle matroid} $M[G]$ of graph $G$ is the matroid whose elements are edges of $G$ and circuits are cycles of $G$. Independent sets and bases of $M[G]$ are forests and maximal forests of $G$ respectively. A matroid is {\em graphic} ({\em simple graphic}) if it is a cycle matroid of a graph (respectively, simple graph).

\vspace{1em}\noindent
Let $E$ be the set of column labels of an $n\times m$ matrix $A$ over a field $F$, and $\mathscr{B}$ be the set of maximal subsets $X$ of $E$ for which the multiset of columns labeled by $X$ is linearly independent in the vector space $F^m$ over $F$. Then the pair $(E,\mathscr{B})$ is the {\em column} ({\em vector}) {\em matroid} of $A$ and is denoted by $M[A]$. In particular, if $F=\Int_2$, then $M[A]$ is a {\em binary matroid}. A binary matroid $M[A]$ is {\em simple} if $A$ does not contain zero columns and no two columns of $A$ are identical (i.e., columns of $A$ are non-zero and distinct). 

\begin{definition}
A binary matroid $M[A]$ is called a {\em \textbf{segment binary matroid}} if $A$ satisfies the consecutive $1$'s property for columns. Moreover, if it is simple, then we call it a {\em \textbf{simple segment binary matroid}}. For any $\emptyset\neq S\subseteq \Int_2^{n-1}$, $M[S]$ denotes the column (vector) matroid of the matrix whose columns are precisely the elements of $S$. Clearly, $M[S]$ is a binary matroid.
\end{definition}

\begin{remark}\label{ssbm cn1}
In particular, when $\emptyset\neq S \subseteq C(n-1)$, $M[S]$ becomes a simple segment binary matroid. So for any simple graph $G$ with $n$ vertices, $M[\beta (G, n)]$ is a simple segment binary matroid. Conversely, every simple segment binary matroid $M[A]$ with $n-1$ rows is same as $M[S]$, where $S$ is the set of column vectors of $A$ over $\Int_2$. Also in this case $S\subseteq C(n-1)$.
\end{remark}

\noindent
Two matroids $M_1=(E_1,\mathscr{B}_1)$ and $M_2=(E_2,\mathscr{B}_2)$ are {\em isomorphic} if there is a bijection $\psi$ from $E_1$ onto $E_2$ such that for all $X\subseteq E_1$, $X$ is independent in $M_1$ if and only if $\psi(X)$ is independent in $M_2$ (or, equivalently, $X$ is a circuit in $M_1$ if and only if $\psi(X)$ is a circuit in $M_2$). In this case, we denote by $M_1\cong M_2$. Also abusing notations we sometimes identify elements of $M[A]$ with its corresponding column vector representation. Thus a simple binary matroid $M[A]$ may be considered as the set of column vectors of $A$. The following theorem characterizes isomorphisms of simple binary matroids in terms of linear transformations. 

\begin{theorem}\label{matiso}
Let $M[A]$ and $M[A_1]$ be two simple binary matroids such that both $A$ and $A_1$ are of same order $n\times m$, $(m,n\in\Nat)$. Then $M[A]\cong M[A_1]$ if and only if there exists a bijective linear operator $T$ on $\Int_2^n$ such that $T$ restricted on $M[A]$ is a bijective map from $M[A]$ onto $M[A_1]$.
\end{theorem}

\begin{proof}
Let $\psi$ be an isomorphism from $M[A]$ onto $M[A_1]$. Let $B$ be a base in $M[A]$. Then $B$ is linearly independent over $\Int_2$. We extend $B$ to a basis $B_1$ (say) of $\Int_2^n$ over $\Int_2$. Now since $\psi$ is an isomorphism, $\psi(B)$ is also a base in $M[A_1]$ and $|\psi(B)|=|B|$. We also extend $\psi(B)$ to $B_2$, a basis of $\Int_2^n$ over $\Int_2$. Then $|B_1\smallsetminus B|=|B_2\smallsetminus \psi(B)|=n-|B|$. Let $f$ be a bijection from $B_1\smallsetminus B$ onto $B_2\smallsetminus \psi(B)$. Now define a map $\Map{T_1}{B_1}{B_2}$ by 
$$T_1(e)=\left\{%
\begin{array}{ll}
\psi(e), & e\in B\\
f(e), & e\in B_1\smallsetminus B
\end{array}\right.$$

\noindent
We next verify that $\psi$ is `linear' on $M[A]$, i.e., if $e_1,e_2$ are columns of $A$ such that $e_1+e_2$ is also a column of $A$, then $\psi(e_1+e_2)=\psi(e_1)+\psi(e_2)$. Let $e=e_1+e_2$. Then $e+e_1+e_2=\mathbf{0}$ which implies that $\set{e,e_1,e_2}$ is a circuit of $M[A]$. Again since $\psi$ is an isomorphism, $\set{\psi(e),\psi(e_1),\psi(e_2)}$ is also a circuit in $M[A_1]$. Hence $\psi(e)+\psi(e_1)+\psi(e_2)=\mathbf{0}$, i.e., $\psi(e_1+e_2)=\psi(e)=\psi(e_1)+\psi(e_2)$. This completes the verification. We extend $T_1$ linearly to obtain a linear operator $T$ on $\Int_2^n$ over $\Int_2$. Then $T$ is bijective as $T_1$ maps a basis bijectively to another basis of $\Int_2^n$ over $\Int_2$ and the restriction of $T$ on $M[A]$ is $\psi$ due to the above verification. 

\vspace{0.5em}\noindent
Conversely, let $T$ be a bijective linear operator on $\Int_2^n$ such that the map $\psi$, the restriction of $T$ on $M[A]$ is a bijective map from $M[A]$ onto $M[A_1]$. Let $X$ be a circuit in $M[A]$. Then $\displaystyle{\sum\limits_{e\in X} e=\mathbf{0}}$ and $\displaystyle{\sum\limits_{e\in A} e\neq \mathbf{0}}$ for all $\emptyset\neq A\subsetneqq X$. Now since $T$ is bijective and linear, we have $\displaystyle{\sum\limits_{e\in A} e = \mathbf{0}}$ if and only if $\displaystyle{\sum\limits_{e\in A} T(e) = \mathbf{0}}$ for all $\emptyset\neq A\subseteq X$. Thus $X$ is a circuit in $M[A]$ if and only if $\psi(X)$ is a circuit in $M[A_1]$. Hence $\psi$ is an isomorphism from $M[A]$ onto $M[A_1]$.
\end{proof}

\begin{corollary}\label{matcong}
Let $M[A]$ and $M[A_1]$ be two simple binary matroids such that both $A$ and $A_1$ are of same order $n\times m$, $(m,n\in\Nat)$. Then $M[A]\cong M[A_1]$ if and only if there exist a non-singular matrix $P$ of order $n\times n$ and a permutation matrix $Q$ of order $m\times m$ such that $PAQ=A_1$.
\end{corollary}

\begin{proof}
If $M[A]\cong M[A_1]$, then following the proof of the direct part of the above theorem, consider two bases $B_1$ and $B_2$ of $\Int_2^n$ over $\Int_2$ and the bijective linear operator $T$ that maps $B_1$ onto $B_2$. Let $P$ be the matrix representation of $T$ with respect to these bases. Then $P$ is a non-singular matrix and $PA=A_2$ where $A_2$ is obtained from $A_1$ by rearranging columns such that $i^{\text{th}}$ column of $A_2$ is the image of the $i^{\text{th}}$ column of $A$ under $T$. Thus $PAQ=A_1$ for some permutation matrix $Q$.

\vspace{0.5em}\noindent
Conversely, let $A_1=PAQ$ for some non-singular matrix $P$ and some permutation matrix $Q$. Let $A_2=A_1Q^{-1}$. Then $PA=A_2$. Since $P$ is non-singular, it corresponds to a bijective linear operator $T$ on $\Int_2^n$ (over $\Int_2$) defined by $T(e)=Pe$ (considering elements of $\Int_2^n$ as column matrices) such that the restriction of $T$ on $M[A]$ is a bijective map from $M[A]$ onto $M[A_2]$. Then $M[A]\cong M[A_2]$ by the above theorem. Since $M[A_1]=M[A_2]$, we have $M[A]\cong M[A_1]$.  
\end{proof}

\noindent
Now we proceed to characterize simple graphic matroids. 

\begin{lemma}\label{gbeta}
Let $G$ be a simple graph with $n$ vertices. Then $M[G] \cong M[\beta (G, n)]$ for any coding sequence $\beta (G, n)$ of $G$.
\end{lemma}

\begin{proof}
It follows from Lemma \ref{kcycle} that cycles of $G$ are precisely the circuits of the matroid $M[\beta (G, n)]$. So $M[G] \cong M[\beta (G, n)]$ as matroids.
\end{proof}

\begin{theorem}\label{grchar}
A matroid is simple graphic if and only if it is isomorphic to a simple segment binary matroid.
\end{theorem} 

\begin{proof}
Let $M$ be a simple graphic matroid. Then $M \cong M[G]$ for a simple graph $G$. By Lemma \ref{gbeta}, we have $M[G] \cong M[\beta (G, n)]$ where $n$ is the number of vertices of $G$. Thus $M$ is isomorphic to a simple segment binary matroid by Remark \ref{ssbm cn1}.

\vspace{0.5em}\noindent
Conversely, let $M[A]$ be a simple segment binary matroid. Then we may consider $M[A]$ as $M[S]$ where $S$ is the set of columns of $A$. By Remark \ref{ssbm cn1}, we have $S \subseteq C(n-1)$, where $A \in M_{n-1, m}(\mathbb Z_2)$. Then by Definition \ref{seqtogrph}, there is a unique simple graph $G$ such that $S=\beta (G, n)$. Therefore, by Lemma \ref{gbeta}, $M[S]=M[\beta (G, n)] \cong M[G]$. Thus, $M[A]$ is a simple graphic matroid.
\end{proof}

\begin{corollary}
A simple binary matroid $M[A]$ $($where $A$ is of order $(n-1)\times m)$ is simple graphic if and only if $m\leqslant \binom{n}{2}$ and there exists a non-singular matrix $P$ such that $PA$ satisfies the consecutive $1$'s property for columns.
\end{corollary}

\begin{proof}
Follows from Theorem \ref{grchar} and Corollary \ref{matcong}.
\end{proof}

\begin{remark}
Since any non-singular matrix is obtained from identity matrix by finite number of elementary row operations, a simple binary matroid $M[A]$ is simple graphic if and only if the consecutive $1$'s property for columns can be obtained from $A$ by finite number of elementary row operations.
\end{remark}

\noindent
It is well known \cite{Ox} that an ordinary matroid isomorphism does not guarantee the corresponding graph isomorphoism for graphic matroids.  We now introduce the concept of a strong isomorphism of simple segment binary matroids.

\begin{definition}\label{strongiso}
Two simple segment binary matroids $M[A_1]$ and $M[A_2]$ are called {\em \textbf{strongly isomorphic}} if
\begin{enumerate}
\item[(1)] $A_1, A_2 \in M_{n-1, m}(\mathbb Z_2)$ for some $m,n \in \mathbb N$, $n\geqslant 2$.
\item[(2)] There exists a bijective linear operator $T$ on $\mathbb Z_2^{n-1}$ such that:
\begin{enumerate}
\item[(i)] $T$ restricted on $C(n-1)$ is a bijection onto itself.
\item[(ii)] $T$ restricted on $M[A_1]$ is a bijective map from $M[A_1]$ onto $M[A_2]$.
\end{enumerate}
\end{enumerate}

\noindent 
We write $M[A_1] \cong _s M[A_2]$ to denote that $M[A_1]$ is strongly isomorphic to $M[A_2]$.
\end{definition}

\noindent
Note that, if $M[A_1] \cong _s M[A_2]$, then the restriction of $T$ on $M[A_1]$ is a matroid isomorphism onto $M[A_2]$ and the restriction of $T$ on $C(n-1)$ is a matroid automorphism. These follow from the fact that $T$ is linear and injective, as then for any subset $X$ of the set of columns of $A_1$, $\sum\limits_{e\in X} e =\mathbf{0}$ if and only if $\sum\limits_{e\in X} T(e) = \mathbf{0}$. In the sequel, we show that strong isomorphism of simple segment binary matroids would guarantee the corresponding graph isomorphism.

\vspace{1em}\noindent
Let $G=(V,E)$ be a (simple undirected) graph with $|V|=n$. Then for any $e\in E$, we use the symbol $p\sim _n |10^x-10^y|$ if $f^*(e)=|10^x-10^y|$ and $p=f^\#(e)\in\beta (G,n)$.

\begin{lemma}\label{p1p2 end}
Let $p_1, p_2$ be distinct elements in $\beta (G, n)$ for any coding sequence $\beta (G,n)$ of a graph $G$ with $n$ vertices. If $p_1 \sim _n |10^x-10^y|$ and $p_2 \sim _n |10^x-10^z|$, then $p_1+p_2 \sim _n |10^y-10^z|$. 
\end{lemma}

\begin{proof}
Let $p_3=p_1+p_2$. So $p_1+p_2+p_3=\mathbf{0}$. From the converse part of the proof of Proposition \ref{3cycle}, we have that $p_3$ corresponds to the end points $10^y$ and $10^z$. Thus, $p_1+p_2 \sim _n |10^y-10^z|$.
\end{proof}

\begin{lemma}\label{path}
Let $S=\{e_1, e_2, \ldots, e_k\} \subseteq \beta (G, n)$ for any coding sequence $\beta (G, n)$ of a graph $G$ with $n$ vertices. Then $\tilde G(S)$ induces a path in $G$ if and only if $\sum \limits_{j=1}^k{e_j} \in C(n-1)$ and $S$ is reduced.
\end{lemma}

\begin{proof}
If $e_1, e_2, \ldots, e_k$ induce a path (in that order) then it is easy to see that we have $e_i \sim _n |10^{x_{i+1}}-10^{x_i}|$ for some distinct $x_1, x_2, \ldots, x_{k+1} \in \mathbb N \cup \{0\}$. By applying Lemma \ref{p1p2 end} repetitively, we have $\sum \limits_{j=1}^k{e_j} \sim _n |10^{x_{k+1}}-10^{x_1}|$. Thus $\sum \limits_{j=1}^k{e_j} \in C(n-1)$. Let $e=\sum\limits_{j=1}^{k} e_j$. Then elements of $S\cup\set{e}$ form a cycle. Then by Lemma \ref{kcycle}, $S\cup\set{e}$ is reduced and so $S$ is reduced.

\vspace{0.5em}\noindent
Conversely, let $\sum \limits_{j=1}^k{e_j} \in C(n-1)$ and $S$ is reduced. Let $\sum \limits_{j=1}^ke_j=e$. So $e+ \sum \limits_{j=1}^k e_j=\mathbf{0}$ and since $S$ is reduced, $S\cup\set{e}$ is also reduced. Now consider the graph $G'$ such that $\beta (G', n)= \beta (G, n) \cup \{e\}$. Clearly, $S \cup \{e\}$ induces a cycle in $\beta (G', n)$ by Lemma \ref{kcycle}. One edge of that cycle corresponds to $e$, all the other edges correspond precisely to the members of $S$. Hence, $\tilde G(S)$ induces a path in $G$.
\end{proof}

\begin{corollary}\label{pqcn1}
Suppose $p \sim _n |10^i-10^j|, q \sim _n |10^r-10^s|$ where $p\neq q$ and $i, j, r, s \in \{0, 1, 2, \ldots, n-1\}$. Then we have $p+q \in C(n-1)$ if and only if $|\{i, j\} \cap \{r, s\}|=1$. Moreover, $p+q \sim_n |10^x -10^y|$, where $x\in\set{i,j}$, $y\in\set{r,s}$ and $x,y\notin\set{i,j}\cap\set{r,s}$.
\end{corollary}

\begin{proof}
Consider a graph $G$ with $n$ vertices such that $p, q \in \beta (G, n)$ for some coding sequence $\beta (G, n)$ of $G$. Clearly, $p$ corresponds to end-points $10^i$ and $10^j$, and $q$ corresponds to end-points $10^r$ and $10^s$. First, let $p+q \in C(n-1)$. So by Lemma \ref{path}, it follows that $\tilde G(\{p, q\})$ induces a path in $G$. This implies that the two edges corresponding to $p$ and $q$ have a common vertex. Since $p \not=q$, this gives that $|\{i, j\} \cap \{r, s\}|=1$.

\vspace{0.5em}\noindent
Conversely, let $|\{i, j\} \cap \{r, s\}|=1$. Suppose $j=r$, without loss of generality. Then by Lemma \ref{p1p2 end}, $p+q \sim _n |10^i-10^s|$ (which proves the next part also). Thus, $p+q \in C(n-1)$.
\end{proof}

\begin{lemma}\label{p1p2p3}
Suppose $p_1 \sim _n |10^i-10^j|, p_2 \sim _n |10^k-10^l|, p_3 \sim _n |10^r-10^s|$, where $i, j, k,l, r, s \in \{0, 1, 2, \ldots, n-1\}$ and $p_1,p_2,p_3$ are distinct. If $p_1+p_2, p_2+p_3, p_1+p_3 \in C(n-1)$, then either $p_1+p_2+p_3 =\mathbf{0}$ or $\{i, j\} \cap \{k, l\}=\{i, j\} \cap \{r, s\}= \set{k,l}\cap\set{r,s}$.
\end{lemma}

\begin{proof}
Consider a graph $G$ with $n$ vertices such that $p_1, p_2, p_3 \in \beta (G, n)$ for some coding sequence $\beta (G, n)$ of $G$. Clearly, $p_1$ corrsponds to end-points $10^i$ and $10^j$, $p_2$ corresponds to end-points $10^k$ and $10^l$ and $p_3$ corresponds to end-points $10^r$ and $10^s$. Now since $p_1+p_2 \in C(n-1)$, by Corollary \ref{pqcn1} we have $|\{i, j\} \cap \{k,l\}|=1$. Without loss of generality, let $j=k$. Then $10^j$ is the common end-point between edges corresponding to $p_1$ and $p_2$. Since we also have that $p_1+p_3 \in C(n-1)$ and $p_2+p_3 \in C(n-1)$, it follows that the edge corresponding to $p_3$ has a common end-point with the edge corresponding to $p_1$ and a common end-point with the edge correponding to $p_2$. If the common end-point in both the cases is $10^j (=10^k)$ then we have $\{i, j\} \cap \{k, l\}=\{i, j\} \cap \{r, s\}=\set{k,l}\cap\set{r,s}$. Otherwise, the common end-point between $p_1$ and $p_3$ must be $10^i$ and the common end-point between $p_2$ and $p_3$ must be $10^l$. Thus, edges corresponding to $p_1, p_2, p_3$ form a cycle involving the vertices $10^i, 10^j, 10^l$. So by Proposition \ref{3cycle}, we have $p_1+p_2+p_3=\mathbf{0}$.  
\end{proof}

\begin{lemma}\label{linear strong}
Let $G, H$ be two simple graphs with $n$ vertices each and suppose $M[\beta (G, n)] \cong _s M[\beta (H, n)]$. Let $T$ be any bijective linear operator on $\Int_2^{n-1}$ giving a strong isomorphism between $M[\beta (G, n)]$ and $M[\beta (H, n)]$. Then for $e_1, e_2 \in C(n-1)$, we have $e_1+e_2 \in C(n-1)$ if and only if $T(e_1)+T(e_2) \in C(n-1)$.
\end{lemma}

\begin{proof}
Let $e_1, e_2 \in \beta (G, n)$. First, let $e_1+e_2 \in C(n-1)$. Now $T(e_1)+T(e_2)=T(e_1+e_2) \in T(C(n-1))=C(n-1)$ as restriction of $T$ maps $C(n-1)$ onto itself. Conversely, let $T(e_1)+T(e_2) \in C(n-1)$. So $T(e_1+e_2) \in C(n-1)$. Again since restriction of $T$ maps $C(n-1)$ onto itself, there exists some $e$ in $C(n-1)$ such that $T(e)=T(e_1+e_2)$. Finally, since $T$ is injective, we have $e=e_1+e_2$. So $e_1+e_2 \in C(n-1)$.
\end{proof}

\noindent
Now we prove the theorem which gives a necessary and sufficient condition for two simple graphs to be isomorphic.

\begin{theorem}\label{isothm}
Let $G$ and $H$ be two simple graphs of $n$ vertices each. Then $G \cong H$ if and only if $M[\beta (G, n)] \cong _s M[\beta(H, n)]$ for any coding sequences $\beta (G, n)$ and $\beta (H, n)$ of $G$ and $H$, respectively.
\end{theorem}

\begin{proof}
We consider vertices of both $G$ and $H$ are labeled by $1, 10, 10^2, \ldots, 10^{n-1}$. First, let $G \cong H$. So there exists a permutation $g$ on the set $\{0, 1, 2, \ldots, n-1\}$ such that for any $r, s \in \{0, 1, 2, \ldots, n-1\}$, we have $10^r$ and $10^s$ are adjacent in $G$ if and only if $10^{g(r)}$ and $10^{g(s)}$ are adjacent in $H$. We consider the graphs $K^{(1)}, K^{(2)}$, where $K^{(1)}=G \cup \bar{G}$ and $K^{(2)}=H \cup \bar{H}$, where $\bar{G}$ and $\bar{H}$ are complements of graphs $G$ and $H$ respectively. Now for each $e \in C(n-1)$, if $e \sim _n |10^i-10^j|$, we define $T(e) \sim _n |10^{g(i)}-10^{g(j)}|$. Clearly, $T$ is a well-defined mapping from $C(n-1)$ into itself, since $|10^p-10^q|$ uniquely determines the pair $\{p, q\}$ for any $p, q$. 

\vspace{1em}\noindent
Now $g$, being a permutation, is a bijection. Suppose $e_1, e_2$ are distinct elements of $C(n-1)$. Let $e_1 \sim _n |10^a-10^b|$ and $e_2 \sim _n |10^c-10^d|$. Clearly, $\{a, b\} \not=\{c, d\}$. Bijectiveness of $g$ implies that $\{g(a), g(b)\} \not=\{g(c), g(d)\}$.  This shows that $T(e_1) \not= T(e_2)$. So we have that $T$ is one-to-one. Since $T$ is a mapping from a finite set into itself, injectiveness of $T$ implies that $T$ is a bijective mapping from $C(n-1)$ onto itself. Again, since $10^r$ and $10^s$ are adjacent in $G$ if and only if $10^{g(r)}$ and $10^{g(s)}$ are adjacent in $H$ for distinct $r, s \in \{0, 1, 2, \ldots, n-1\}$, we have that $T$ restricted on $\beta (G, n)$ is a mapping from $\beta (G, n)$ into $\beta (H, n)$. Injectiveness of $T$ ensures the injectiveness of $T$ restricted to $\beta (G, n)$. Since $\beta (G, n)$ and $\beta (H, n)$ are finite sets with equal cardinality, we have that $T$ restricted to $\beta (G, n)$ is a bijection from $\beta (G, n)$ onto $\beta (H, n)$. 

\vspace{1em}\noindent
Next we observe that $T$ is `linear' on $C(n-1)$, i.e., if $p,q\in C(n-1)$ such that $p+q\in C(n-1)$, then $T(p)+T(q)=T(p+q)$. Let $p\sim_n |10^i-10^j|$, $q\sim_n |10^r-10^s|$ such that $p+q\in C(n-1)$. Then by Corollary \ref{pqcn1}, we have $|\set{i,j}\cap\set{r,s}|=1$. Without loss of generality, we assume $j=r$. So $g(j)=g(r)$ and $p+q\sim_n |10^i-10^s|$. Now $T(p)\sim_n |10^{g(i)}-10^{g(j)}|$ and $T(q)\sim_n |10^{g(r)}-10^{g(s)}|$. Since $g(j)=g(r)$, we have $T(p)+T(q)\in C(n-1)$ and $T(p)+T(q)\sim_n |10^{g(i)}-10^{g(s)}|$. Since $T(p+q)\sim_n |10^{g(i)}-10^{g(s)}|$, we have $T(p+q)=T(p)+T(q)$.

\vspace{1em}\noindent
Now let $e_i \sim _n |10^{i}-10^{i-1}|$ for all $i=1, 2, \ldots, n$. We know that $B=\{e_i \mid i=1, 2, \ldots, n\}$ is a basis of $\mathbb Z_2^{n-1}$. Since $T$ is defined on each $e_i$ as the latter is in $C(n-1)$ (in fact, $T(e_i) \sim _n |10^{g(i)}-10^{g(i-1)}|$), we can extend $T$ linearly to a linear operator $T_1$ on $\mathbb Z_2^{n-1}$. Bijectiveness of $T_1$ follows from injectiveness of $T$ on $C(n-1)$ (which ensures distinct images under $T$ for distinct elements of $B$, thus ensuring injectiveness of $T_1$) and finiteness of $\mathbb Z_2^{n-1}$. As $M[\beta (G, n)]$ and $M[\beta (H, n)]$ are also of the same order, we have $M[\beta (G, n)] \cong _s M[\beta (H, n)]$. 

\vspace{1em}\noindent
Conversely, let $M[\beta (G, n)] \cong _s M[\beta(H, n)]$. So there exists a bijective linear operator $T$ satisfying the properties mentioned in the Definition \ref{strongiso}. We find a permutation $g$ on the set $\{0, 1, 2, \ldots, n-1\}$ such that for any distinct $r, s \in \{0, 1, 2, \ldots, n-1\}$, $10^r$ and $10^s$ are adjacent in $G$ if and only if $10^{g(r)}$ and $10^{g(s)}$ are adjacent in $H$. Clearly, such a $g$ acts as an isomorphism between $G$ and $H$. Now for $i=1, 2, \ldots, n$, define $e_i$ as the element in $C(n-1)$ which has 1 in its $i^{\text{th}}$ co-ordinate from the right ($(n-i)^\text{th}$ co-ordinate from the left) and 0 in remaining coordinates. Then $e_i \sim _n |10^i-10^{i-1}|$. Now for $i=1,2,\ldots,n-1$, we have $e_i+e_{i+1} \in C(n-1)$. So from Lemma \ref{linear strong}, we have $T(e_i)+T(e_{i+1}) \in C(n-1)$. From Corollary \ref{pqcn1}, the edges corresponding to $T(e_i)$ and $T(e_{i+1})$ have a (unique) common point, say $10^x$. We define $g(i-1)=x$. This defines $g$ as a mapping from $\set{0,1,2,\ldots,n-2}$ into $\{0, 1, 2, \ldots, n-1\}$. Now we show that $g$ is one-to-one.

\vspace{1em}\noindent
If possible, let $g(i-1)=g(j-1)$ for some $i \not=j$, $i,j\in\set{1,2,\ldots,n-1}$. By the above definition of $g$, $g(i-1)$ and $g(j-1)$ are one of the common end points of $T(e_i), T(e_{i+1})$ and $T(e_j),T(e_{j+1})$ respectively. So the edges corresponding to $T(e_i)$ and $T(e_j)$ have a common end-point $g(i-1)$ ($=g(j-1)$). So from Corollary \ref{pqcn1}, we have $T(e_i)+T(e_j) \in C(n-1)$. Linearity of $T$ implies that $T(e_i+e_j) \in C(n-1)$. Since $T$ restricted to $C(n-1)$ is a bijection from $T$ onto itself, we have $e_i+e_j \in C(n-1)$.  Again, $e_j+e_{i+1} \in C(n-1)$ for the same reason as $g(i-1)$ is a common end-point between edges corresponding to $T(e_j)$ and $T(e_{i+1})$. We also have $e_i+e_{i+1} \in C(n-1)$. So $e_i+e_j, e_{i+1}+e_i, e_j+e_{i+1} \in C(n-1)$. If possible, let $e_i+e_j+e_{i+1}=\mathbf{0}$. However, then $e_j=e_i+e_{i+1}$, which is impossible by definition of $e_i$'s. Thus, from Lemma \ref{p1p2p3}, $10^i$ is either $10^j$ or $10^{j-1}$, i.e., $i=j$ or $j-1$. Similar argument on $e_i,e_j,e_{j+1}$ gives us $j=i$ or $i-1$. But $i=j-1$ and $j=i-1$ both cannot be true simultaneously. So we have $i=j$ which is a contradiction. Thus $g$ is injective. So we define $g(n-1)=x$, where $x\in\set{0,1,2,\ldots,n-1}\smallsetminus\set{g(0),g(1),\ldots,g(n-2)}$. Then $g$ is defined on $\set{0,1,2,\ldots,n-1}$ into itself. Moreover since $g$ is injective, it is a permutation on the set $\{0, 1, 2, \ldots, n-1\}$.
\end{proof}


\bibliographystyle{amsplain}

\end{document}